\documentclass{article}
\usepackage[utf8]{inputenc}
\usepackage{amssymb, amsmath, xcolor}
\usepackage{amsthm}
\usepackage{url}
\usepackage{hyperref}
\usepackage{tikz,subcaption}
\usepackage[enableskew]{youngtab}
\usepackage{ytableau,varwidth}
\usepackage[hang,flushmargin]{footmisc} 

\definecolor{red}{rgb}{1,0,0}
\definecolor{blue}{rgb}{.2,.2,.8}
\definecolor{magenta}{rgb}{1,0,1}
\definecolor{dartmouthgreen}{rgb}{0.05, 0.5, 0.06}
\definecolor{purple(x11)}{rgb}{0.63,0.36,0.94}
\definecolor{turquoise}{rgb}{0.25, 0.87, 0.81}

\def\m{\mu}

\def\spmod{\!\!\!\!\pmod}

\newtheorem{theorem}{Theorem}[section]
\newtheorem{corollary}[theorem]{Corollary}
\newtheorem{proposition}[theorem]{Proposition}

\newtheorem{conj}{Conjecture}
\theoremstyle{definition}

\newtheorem{remark}{Remark}

\title{On a Partition Identity of Lehmer}

\author{Cristina Ballantine, Hannah Burson, Amanda Folsom, \\Chi-Yun Hsu, Isabella Negrini, and Boya Wen}

\date{}
 \begin{document}

\maketitle

\begin{abstract}
Euler's identity equates the number of partitions of any non-negative integer $n$ into odd parts and the number of partitions of $n$ into distinct parts.
Beck conjectured and Andrews proved the following companion to Euler's identity:  
the excess of the number of parts in all partitions of $n$ into odd parts over the number of parts in all partitions of $n$ into distinct parts equals the number of partitions of $n$ with exactly one even part (possibly repeated).  Beck's original conjecture was followed by generalizations and so-called ``Beck-type" companions to other identities.

In this paper, we establish a collection of Beck-type companion identities to the following result mentioned by Lehmer at the 1974 International Congress of Mathematicians:  the excess of the number of partitions of $n$ with an even number of even parts over the number of partitions of $n$ with an odd number of even parts equals the number of partitions of $n$ into distinct, odd parts.    We also establish various generalizations of Lehmer's identity, and prove related Beck-type companion identities.  We use both analytic and combinatorial methods in our proofs.

\end{abstract}

\renewcommand{\thefootnote}{\fnsymbol{footnote}} 
\renewcommand{\thefootnote}{\arabic{footnote}} 

\section{Introduction and statement of results}\label{sec_intro}

Many results in the theory of partitions  concern  identities asserting that the set $\mathcal P_{X}(n)$ of partitions of $n$ satisfying condition $X$ and the set $\mathcal P_{Y}(n)$ of partitions of $n$ satisfying condition $Y$ are equinumerous. Likely the oldest such result is Euler's identity that the number of partitions of $n$ into odd parts is equal to the number of partitions of $n$ into distinct parts. In 2017, Beck 
 made the following conjecture (\cite{oeisA090867}, \cite[Conjecture]{A17}):
\begin{conj}[Beck] \label{conj}
The excess of the number of parts in all partitions of $n$ into odd parts over the number of parts in all partitions of $n$ into distinct parts equals the number of partitions of $n$ with exactly one even part (possibly repeated).
\end{conj}  
Beck's conjecture was quickly proved analytically by Andrews \cite{A17}, who additionally showed that this excess also equals the number of partitions of $n$ with exactly one  part  repeated (and all other parts distinct).   
The conjecture was also proved combinatorially by Yang \cite{Yang19} and Ballantine--Bielak \cite{BB19} independently.
This  work was followed by  generalizations  and Beck-type  companions to other well known identities (e.g., \cite{AB19},  \cite{BW21}, \cite{LW20}, \cite{Yang19}).  
In general, a Beck-type companion identity to  $|\mathcal P_{X}(n)|=|\mathcal P_{Y}(n)|$ is an identity that equates the excess of the number of parts in all partitions in $\mathcal P_{X}(n)$ over the number of parts in all partitions in $\mathcal P_{Y}(n)$ to the number of partitions of $n$ satisfying a condition closely related to $X$ (or $Y$). 

In this article, we establish a number of Beck-type identities related to a result of Lehmer, which he informally mentioned at the 1974 International Congress of Mathematicians \cite{G75}:  for every non-negative integer $n$, we have that 
\begin{equation} \label{lehmer1} 2p_e(n,2)=p(n)+q_o(n),
\end{equation} where $$p_{e}(n,2) := p(n \ | \ \text{ the number of even parts is even})$$  and $$q_o(n) := p(n \ | \ \text{ distinct, odd parts}).$$
 Here and throughout we use the  standard notations $p(n)$ and $p(n \ | \  X)$ to denote the number of partitions of $n$, and the number of partitions of $n$ satisfying condition $X$, respectively.  If we also denote by $$p_{o}(n,2) := p(n \ | \ \text{ the number  of even parts is odd}),$$ identity \eqref{lehmer1} is equivalent to the    following statement which we refer to as Lehmer's identity. 
\begin{theorem} \label{lehmer-thm} For any $n\in \mathbb N_0:=\mathbb N \cup \{0\}$, we have \begin{equation}\label{lehmer} p_e(n,2) = p_o(n,2) + q_o(n).\end{equation}
\end{theorem}
An analytic proof of Theorem \ref{lehmer-thm} is immediate: The generating series for $p_e(n,2)-p_o(n,2)$ and $q_o(n)$ are ${(q;q^2)^{-1}_\infty(-q^2;q^2)^{-1}_\infty}$  and $(-q;q^2)_\infty$, respectively. Then Theorem \ref{lehmer-thm} follows from the fact that  $$(-q;q^2)_\infty=\frac{(-q;q)_\infty}{(-q^2;q^2)_\infty}$$ and Euler's identity $$(-q;q)_\infty=\frac{1}{(q;q^2)_\infty}.$$  Here and throughout, the $q$-Pochhammer symbol is given by \begin{align*}
& (a;q)_n := \begin{cases}
1, & \text{for $n=0$,}\\
(1-a)(1-aq)\cdots(1-aq^{n-1}), &\text{for $n>0$;}
\end{cases}\\
& (a;q)_\infty := \lim_{n\to\infty} (a;q)_n.
\end{align*}
In \cite{G75}, Gupta   provided a beautiful combinatorial proof of Theorem \ref{lehmer-thm}.   We also note that \eqref{lehmer} is equivalent to the following identity due to  Glaisher (\cite[p.129]{D52} \cite[p.256]{G1876}) 
$$p_e(n)-p_o(n)=(-1)^nq_o(n),$$ where $$p_{e/o}(n) := p(n \ | \ \text{even/odd number of parts}).$$ 

Our first main result, Theorem \ref{beck-lehmer0}  below, is a Beck-type companion identity to Lehmer's identity \eqref{lehmer}.  To state it, we first set some additional notation.  We begin by formally defining a \emph{partition}   $\lambda=(\lambda_1, \lambda_2, \ldots, \lambda_j)$ of \emph{size} $n\in\mathbb N_0$  to be a non-increasing sequence of positive integers $\lambda_1\geq \lambda_2 \geq \cdots \geq \lambda_j$ called \emph{parts} that add up to $n$.  For convenience, we abuse notation and use $\lambda$ to denote either the multiset of its parts or the non-increasing sequence of parts.  We write $a\in \lambda$ to mean the positive integer $a$ is a part of $\lambda$. The empty partition is the only partition of size $0$. Thus, $p(0)=1$. We write $|\lambda|$ for the size of $\lambda$ and  $\lambda\vdash n$  to mean that $\lambda$ is a partition of size $n$. For a pair of partitions $(\lambda, \mu)$ we also write $(\lambda, \mu)\vdash n$ to mean $|\lambda|+|\mu|=n$. 
We use the convention that $\lambda_k=0$ for all $k$ greater than the number of parts. When convenient we will also use the exponential notation for parts in a partition:  the exponent of a part is the multiplicity of the part in the partition, e.g., we write $(a^b)$ for the partition consisting of $b$ parts equal to $a$. 
 Further, we denote by calligraphy style capital letters the set of partitions enumerated by the function denoted by the same letter. For example,  $\mathcal Q_o(n)$ denotes the set of partitions of $n$ into distinct odd parts. We also define $\mathcal Q_o := \bigcup_{n\geq 0}  \mathcal Q_o(n)$.

 \begin{theorem}\label{beck-lehmer0} Let $n\in \mathbb N_0$. The excess of the  number of parts in all partitions in $\mathcal P_e(n,2)$ over the number of parts
 in all partitions in $ \mathcal P_o(n,2)\cup \mathcal Q_o(n)$    equals the number of partitions of $n$ with exactly one even part, possibly repeated, and all other parts odd and distinct. 
 \end{theorem}
 
 \begin{remark}
  As proved in \cite{AB19}, the excess in Theorem \ref{beck-lehmer0} is almost always equal to the number of parts in all self-conjugate partitions of $n$.  Hence, the excess  in the number of parts in all partitions in $\mathcal P_e(n,2)$ over the number of parts in all partitions in $\mathcal P_o(n,2)$ is almost always equal to the total   number of parts in all self-conjugate partitions of $n$  and in all partitions of $n$ into distinct odd parts.  More precisely, if $N(x)$ is the number of times the above statement is true for $n\leq x$, then $\lim_{x\to \infty}{N(x)}/{x}=1$.

 \end{remark}

 We also establish a \emph{restricted} Beck-type identity accompanying \eqref{lehmer} in which we only count the number of even parts in partitions in $\mathcal P_e(n,2)$ and $\mathcal P_o(n,2)$; this result is given in Theorem \ref{beck-lehmer} below.   To ease notation in the statement of this result and  other Beck-type identities that follow, we introduce the following definition.   
Let $n, r, a, b$ be non-negative integers such that $1\leq ab\leq n$.  
 We define 
 \begin{align*} \mathcal B_r(n,a,b) := \left \{ \lambda \vdash n-rab \ \Bigg | \ \begin{array}{l}
 \lambda\neq (ra, r(a-2)), \text{ and }  \\ r(a+b+1)\not \in \lambda, \text{ and}  \\ \lambda_1-\lambda_2 \leq 2r(a+b+1)  \text{ or } \lambda_1 = 3r(a+b+1)
 \end{array} \!\!\right\}.
 \end{align*}
We write $\mathcal B(n,a,b)$ for $\mathcal B_1(n,a,b)$.
 
\begin{theorem}\label{beck-lehmer} Let $n\in \mathbb N_0$. The excess of the number of parts in all partitions in $\mathcal Q_o(n)$ plus the number of even parts in all partitions in $\mathcal P_o(n,2)$ over the  number of even parts in all partitions in $\mathcal P_e(n,2)$ equals the number of pairs of partitions $(\lambda, (a^b))$ satisfying the following conditions: 
\begin{itemize}
    \item[i.]  
    $a,b$ are both odd,
    \item[ii.] $\lambda\in \mathcal Q_o\cap \mathcal B(n,a,b)$, i.e., $\lambda$ has distinct odd parts,   is not equal to $(a,a-2)$, does not have $a+b+1$ as a part, and satisfies $\lambda_1-\lambda_2\leq 2(a+b+1)$ or $\lambda_1 = 3(a+b+1)$.
\end{itemize}
\end{theorem}

\begin{remark} If $n$ is even, the condition $\lambda \neq (a, a-2)$ in $ii.$ is vacuously true. \end{remark}

In general, whenever we refer to pairs of the form $(\lambda, (a^b))$, we require $(a^b)$ to be nonempty (i.e. $a,b>0$), while $\lambda$ is allowed to be the empty partition.

\begin{remark} \label{rmk_Beckpair} Beck's Conjecture \ref{conj} can also be formulated in the language of pairs as in Theorem \ref{beck-lehmer}:

The excess of the number of parts in all partitions of $n$ into odd parts over the number of parts in all partitions of $n$ into distinct parts equals the number of pairs of partitions $(\lambda, (a^b))\vdash n$ satisfying the following conditions: 

\begin{itemize}
    \item[i.] 
    $a$ is even,
    \item[ii.] $\lambda$ is a partition  into odd parts.
\end{itemize}

\end{remark}

Next, we give a collection of Beck-type companion identities to the following generalization of 
Lehmer's identity \eqref{lehmer}, which we prove in Section \ref{sec_proofs1}. For the remainder of the paper, we let $r \in \mathbb N$. 

\begin{theorem}\label{lehmer-gen} For any $n\in \mathbb N_0$,  we have
\begin{equation} \label{lehmer-g}p_{e}(n,2r)
=  p_{o}(n,2r) + q_o(n,r), 
\end{equation} where
\begin{align*}
p_{e/o}(n,2r) &:= p(n \ | \  \text{all parts allowed, even/odd no.\ of parts divisible by $2r$ }) \\ 
q_o(n,r) &:=p\left(n \ \Big| \ \begin{array}{l}\text{parts are not divisible by $2r$,} \\ \text{parts divisible by $r$ are distinct}\end{array}\right)\\
&=p\left(n  \Big|  \begin{array}{l}\text{all parts divisible by $r$ are distinct, odd multiples of $r$} \end{array}\right).
\end{align*}
\end{theorem} 
\noindent Note that for $r=1$, identity~\eqref{lehmer-g} reduces to identity \eqref{lehmer}.

Our first Beck-type companion identity to \eqref{lehmer-g} is given by the next theorem which becomes Theorem \ref{beck-lehmer0} when $r=1$.

\begin{theorem}\label{beck-lehmer2prime}
Let $n\in \mathbb N_0$.  The excess in the total number of parts in all partitions in $\mathcal P_e(n,2r)$ over the total number of parts in all partitions in $\mathcal P_o(n,2r)\cup \mathcal Q_o(n,r)$ equals the number of  pairs of partitions $(\lambda, (a^b))\vdash n$ such that 
\begin{itemize}
    \item[i.]  $2r \mid a$,
    \item[ii.] $\lambda\in \mathcal{Q}_o(n-ab,r)$.
\end{itemize}
\end{theorem}

\begin{remark}
Equivalently, the excess of Theorem \ref{beck-lehmer2prime} equals the number of partitions of $n$ in which, among the parts divisible by $r$, there is a single even multiple of $r$ and this part is possibly repeated, while all other parts divisible by $r$ are odd multiples of $r$ and they are distinct.
\end{remark}

Theorem \ref{beck-lehmer2} below is a restricted Beck-type companion identity to \eqref{lehmer-g}, in which we only count the number of parts divisible by $r$ in $\mathcal{Q}_o(n,r)$, and the number of parts divisible by $2r$ in $\mathcal{P}_e(n,2r)$ and $\mathcal{P}_o(n,2r)$.
The theorem reduces to Theorem
\ref{beck-lehmer} when $r=1$.

\begin{theorem}\label{beck-lehmer2} 
Let $n\in \mathbb N_0$. The excess of the number of parts divisible by $r$ in all partitions in $\mathcal Q_o(n,r)$ plus the number of parts divisible by $2r$ in all partitions in $\mathcal P_o(n,2r)$ over the  number of parts divisible by $2r$ in all partitions in $\mathcal P_e(n,2r)$ equals the number of pairs of partitions $(\lambda,((ar)^b))$ satisfying the following conditions:
\begin{itemize}
    \item[i.] 
    $a,b$ are both odd,
    \item[ii.]
    $\lambda\in \mathcal Q_o(n-rab,r)$ such that, if we write $\lambda=\lambda^{ndiv}\cup \lambda^{div}$ where $\lambda^{div}$ contains all parts of $\lambda$ that are divisible by $r$, then $\lambda^{div}\in \mathcal B_r(n-|\lambda^{ndiv}|, a, b)$.

\end{itemize}
\end{theorem}
\noindent Recall that $\lambda\cup \mu$  is the partition whose  parts are precisely the parts of $\lambda$ and $\mu$ (with multiplicities).

Next we give another generalization of Lehmer's identity \eqref{lehmer}.    To describe this, we let $r\in\mathbb N$, and   let $  L_r\subseteq \{2,4,6,\dots,2r\}$, with $L_r\neq \varnothing$.   We use the sets $L_r$ to restrict even parts of partitions to lie within certain arithmetic progressions.  More precisely, we
define
\begin{align*}
{p}_{e/o}&(n,  L_r,2r) 
\\ &:= p\left(n \ \Bigg| \  \begin{array}{l}\text{all parts allowed,} \\ \text{even parts $\equiv \ell\spmod{2r}, \ell \in   L_r$,} \\\text{even/odd no.\ of even parts}\end{array}\right),  \\  
{q}(n,  & \ L_r, r) \\ &:=p\left(n \ \Big| \ \begin{array}{l}\text{all parts distinct,} \\\text{even parts   $\not \equiv \ell\spmod{2r},$   $\ell \in  L_r$}\end{array}\right).
\end{align*}

\begin{theorem}\label{lehmer-gen2}
For any $n\in\mathbb N_0$, we have  
\begin{align}\label{lehmer-g2} &{p}_{e}(n,  L_r,2r)
=  {p}_{o}(n,  L_r,2r) + {q}(n, L_r,r).  
\end{align}   
  \end{theorem}
\noindent Note that in the case $L_r = \{2,4,\ldots,2r\}$, identity \eqref{lehmer-g2} is equivalent to identity \eqref{lehmer}.

The next theorem is a Beck-type companion identity to \eqref{lehmer-g2}, which becomes Theorem \ref{beck-lehmer0} when $L_r = \{2,4,\ldots,2r\}$.
\begin{theorem}\label{beck-lehmer3prime}  Let $n\in \mathbb N_0$. The excess in the total number of parts in all partitions in $\mathcal P_e(n,L_r,2r)$ over the total number of parts in all partitions in $\mathcal P_o(n,L_r,2r)\cup \mathcal Q(n,L_r,r)$ equals
the number of pairs of partitions $(\lambda,(a^b))$ satisfying the following conditions: 
\begin{itemize}
    \item[i.] $a$ is  even
    \item[ii.] $\lambda\in \mathcal{Q}(n-ab,L_r,r)$.
\end{itemize}
\end{theorem}

A restricted Beck-type companion identity to \eqref{lehmer-g2} is given by the next theorem, where we only count the number of even parts in $\mathcal{P}_e(n,L_r,2r)$ and $\mathcal{P}_o(n,L_r,2r)$.
The theorem becomes Theorem \ref{beck-lehmer} when $L_r = \{2,4,\ldots,2r\}$.

\begin{theorem} \label{beck-lehmer3}
Let $n\in \mathbb N_0$. The excess of the number of parts in all partitions in $\mathcal Q(n,L_r,r)$ plus the number of even parts in all partitions in $\mathcal P_o(n,L_r,2r)$ over the  number of even parts in all partitions in $\mathcal P_e(n,L_r,2r)$ equals the number of pairs of partitions $(\lambda, (a^b))$ satisfying the following conditions:
\begin{enumerate}
    \item[i.] 
    $a,b$ are both odd,
    \item[ii.] $\lambda\in \mathcal{Q}(n-ab,L_r,r)$ such that, if we write $\lambda=\lambda^e\cup\lambda^o$, where $\lambda^e$ consists of all the even parts of $\lambda$ and $\lambda^o$ consists of all the odd parts of $\lambda$, then $\lambda^o\in \mathcal B(n-|\lambda^e|, a, b)$.

\end{enumerate}
\end{theorem}

The next result is a new restricted Beck-type companion identity to Lehmer’s identity \eqref{lehmer}, different from Theorem \ref{beck-lehmer}.
We only count the number of parts in certain arithmetic progressions in $\mathcal{Q}_o(n,2)$, $\mathcal{P}_e(n,2)$ and $\mathcal{P}_o(n,2)$.

To describe it, for $r\in \mathbb N$, let
\begin{align*} 
  L_r\subseteq \{2,4,6,\dots,2r\},\ \ \ 
  O_r \subseteq\{1,3,5,\dots,2r-1\}.  
\end{align*}

\begin{theorem}\label{beck-lehmer4}
 Let $n$ be a positive integer, and $L_r$ and $O_r$ as above such that if $n\equiv 0\pmod 4$ then $2 \not \in L_r$.  The excess of the number of parts $\equiv \ell \pmod{2r}, \ell \in O_r$ in all partitions in $ {\mathcal Q}_o(n)$ plus the number of parts  $\equiv \ell \pmod{2r}, \ell \in L_r$  in all partitions in $\mathcal P_o(n,2)$ over the  number of parts  $\equiv \ell \pmod{2r}, \ell \in L_r$ in all partitions in $\mathcal P_e(n,2)$   equals the number of pairs of partitions $(\lambda, (a^b))\vdash n$ satisfying the following conditions:
\begin{enumerate}
    \item[i.] 
     $a \equiv \ell \pmod{2r}$ for some $\ell \in L_r\cup O_r$, and $b$ is odd. Moreover, if $a$ is odd, then $b=1$,
    \item[ii.] $\lambda\in \mathcal{Q}_o$. Moreover, if $a$ is odd, then $a\not \in \lambda$; if $a$ is even, then $\lambda_1-\lambda_2\leq a$ and $\lambda \not \in \{(\frac{a}{2}+1, \frac{a}{2}-1),(\frac{a}{2}+2, \frac{a}{2}-2) \}$. 
\end{enumerate}

  If $n\equiv 0\pmod 4$ and $2 \in L_r$, the excess is one less than the number of pairs counted above. Moreover, if additionally $n\not\in \{4,8,12,16,20\}$, then the excess is equal to the number of pairs $(\lambda, (a^b))$ satisfying i. and ii.  with the additional condition $(\lambda, (a^b)) \neq ((9,7,5,1),(2^b))$.

\end{theorem}

\begin{remark}
If $n \not\equiv 0 \pmod 4$, then the condition $\lambda \not \in \{(\frac{a}{2}+1, \frac{a}{2}-1),(\frac{a}{2}+2, \frac{a}{2}-2) \}$ is vacuously true. 
\end{remark}

Generally speaking, our proofs are both analytic and combinatorial in nature.    
In Sections \ref{sec_newproof} to \ref{sec_proofs2}, we prove Theorems \ref{beck-lehmer0} through \ref{beck-lehmer3}. 
In Section \ref{sec_proofs3}, 
 we provide two paths to prove Theorem \ref{beck-lehmer4} and give several important examples.  The first proof relies upon the non-negativity of certain  $q$-series coefficients and their combinatorial interpretation, while the second proof establishes a relevant combinatorial injection.
In Section \ref{sec_aux}, we establish the non-negativity of the coefficients of some related $q$-series.

\allowdisplaybreaks
 
\section{Proofs of Theorems \ref{beck-lehmer0} and \ref{beck-lehmer}}\label{sec_newproof}
Consider the generating series

\begin{multline}\nonumber F(z;q):=\frac{1}{(zq;q^2)_\infty(-zq^2;q^2)_\infty}\\ {{=\sum_{n=0}^\infty\sum_{m=0}^\infty \sum_{s=0}^m p(n \ | \ m \text{ parts, of which  } s  \text{ parts are even})(-1)^sz^m q^n, 
}}
\end{multline} 
\begin{multline}\nonumber
    E(z;q):=\frac{1}{(q;q^2)_\infty(-zq^2;q^2)_\infty}\\ {{=\sum_{n=0}^\infty\sum_{m=0}^\infty p(n \ | \ \text{the number of even parts is } m)(-z)^m q^n, 
}}
\end{multline}
and   \begin{multline}\nonumber Q_o(z;q):=(-zq;q^2)_\infty \\ 
{{=\sum_{n=0}^\infty \sum_{m=0}^\infty 
p(n \ | \ \text{parts must be odd and distinct, } m \text{ parts}) z^m q^n
}}
.\end{multline}

To prove Theorem \ref{beck-lehmer0}, note that $\frac{\partial}{\partial z} \big |_{z=1}(F(z;q)  -Q_o(z;q))$ gives the generating series for the excess of the  number of parts in all partitions in $\mathcal P_e(n,2)$ over the number of parts
 in all partitions in $\mathcal Q_o(n) \cup \mathcal P_o(n,2)$.
 We have
 \begin{align*}\frac{\partial}{\partial z} \Big |_{z=1}(F(z;q) & -Q_o(z;q)) \\ 
&=(-q;q^2)_\infty\left(\sum_{k=0}^\infty \frac{q^{2k+1}}{1-q^{2k+1}}-\sum_{k=1}^\infty \frac{q^{2k}}{1+q^{2k}}-\sum_{k=0}^\infty   \frac{q^{2k+1}}{1+q^{2k+1}} \right)\\ 
&=(-q;q^2)_\infty\left(\sum_{k=0}^\infty \frac{q^{2k+1}}{1-q^{2k+1}}-\sum_{k=1}^\infty \frac{q^k}{1+q^{k}}\right)\\
&= (-q;q^2)_\infty\left(\sum_{k=0}^\infty \frac{q^{2k+1}}{1-q^{2k+1}}-\sum_{k=1}^\infty \frac{q^k}{1-q^{2k}}+\sum_{k=1}^\infty \frac{q^{2k}}{1-q^{2k}}\right)\\
&= (-q;q^2)_\infty\left(\sum_{k=1}^\infty \frac{q^{k}}{1-q^{k}}-\sum_{k=1}^\infty \frac{q^k}{1-q^{2k}}\right)\\
&= (-q;q^2)_\infty\sum_{k=1}^\infty \frac{q^{2k}}{1-q^{2k}}.
\end{align*}
The last expression is the generating series for the number of partitions of $n$ with exactly one even part, possibly repeated, and all other parts odd and distinct.   This proves Theorem \ref{beck-lehmer0}.  

To prove Theorem \ref{beck-lehmer}
we  note that      $\frac{\partial}{\partial z} \big |_{z=1}(Q_o(z;q)-E(z;q))$ is the generating series for the excess   of the number of parts in all partitions in $\mathcal Q_o(n)$ plus the number of even parts in all partitions in $\mathcal P_o(n,2)$ over the  number of even parts in all partitions in $\mathcal P_e(n,2)$. 
We compute  \begin{align*}\frac{\partial}{\partial z} \Big |_{z=1}(Q_o(z;q)&-E(z;q))\\
&=(-q;q^2)_\infty\left(\sum_{k=0}^\infty   \frac{q^{2k+1}}{1+q^{2k+1}}+\sum_{k=1}^\infty \frac{q^{2k}}{1+q^{2k}}\right)\\
&=(-q;q^2)_\infty \sum_{k=1}^\infty \frac{q^k}{1+q^k} \\
&= (-q;q^2)_\infty \left(\sum_{k=1}^\infty \frac{q^k}{1-q^{2k}}-\sum_{k=1}^\infty \frac{q^{2k}}{1-q^{2k}}\right). 
\end{align*} 

Let 
\begin{align*}
    p_{e^o}(n):=p(n \ | \ \text{odd number of identical even parts}),
\end{align*}
i.e.,
\begin{align*}
    p_{e^o}(n):=|\{\lambda\vdash n \mid  \lambda= (a^b), a \mbox{ even and } b \mbox{ odd} \}|.
\end{align*}
Define $p_{e^e}(n),p_{o^e}(n)$, and $p_{o^o}(n)$ similarly.

Then
\begin{align*}
    \sum_{k=1}^\infty \frac{q^k}{1-q^{2k}}=\sum_{n=1}^{\infty} (p_{o^o}(n)+p_{e^o}(n)) q^n,
\end{align*}
and
\begin{align*}
    \sum_{k=1}^\infty \frac{q^{2k}}{1-q^{2k}}=\sum_{n=1}^{\infty} (p_{o^e}(n)+p_{e^e}(n)) q^n.
\end{align*}
Since conjugation gives a bijection between $\mathcal{P}_{o^e}(n)$ and $\mathcal{P}_{e^o}(n)$, we further have
\begin{align*}
    \sum_{k=1}^\infty \frac{q^k}{1-q^{2k}}-\sum_{k=1}^\infty \frac{q^{2k}}{1-q^{2k}}
    =\sum_{n=1}^{\infty} (p_{o^o}(n)-p_{e^e}(n)) q^n.
\end{align*}
Therefore  
\begin{align*}(-q;q^2)_\infty \sum_{k=1}^\infty \frac{q^k}{1+q^k}
=&\left(\sum_{n=0}^{\infty} q_o(n) q^n\right)\left(\sum_{n=1}^{\infty} (p_{o^o}(n)-p_{e^e}(n)) q^n\right)\\
=&\sum_{n=1}^{\infty}\left(\sum_{m=0}^{n-1}q_o(m)p_{o^o}(n-m)-q_o(m)p_{e^e}(n-m)\right)q^n,
\end{align*} 
and the excess in question is given by  $$\sum_{m=0}^{n-1}\left(q_o(m)p_{o^o}(n-m)-q_o(m)p_{e^e}(n-m)\right).$$ 

Equivalently, this is the excess of the number of elements in
$$B(n):=\{(\lambda,(a^b))\vdash n \ | \ \lambda \in \mathcal Q_o,  a, b \text{ odd}\}$$ 
over that in
$$A(n):=\{(\lambda,(a^b))\vdash n\ | \ \lambda \in \mathcal Q_o,  a, b \text{ even}\}.$$
To measure this excess, we construct an injection $T$ from $A(n)$ to $B(n)$ as follows. We partition the set $A(n)$ into three disjoint subsets:
\begin{align*}
    &A_1(n):=\{(\lambda,(a^b))\in A(n) \mid a+b-1 \not \in \lambda \};\\
    &A_2(n):=\{(\lambda,(a^b))\in A(n) \mid a+b-1  \in\lambda  \text{ and $\lambda$ has at least two parts}\};\\
    &A_3(n):=\{(\lambda,(a^b))\in A(n) \mid \lambda=(a+b-1) \}.
\end{align*}
We define $T$ on each $A_i(n)$ in the following way.

\begin{enumerate}
    \item If $(\lambda, (a^b))\in A_1(n)$ (including the case where $\lambda$ is empty), then $$T(\lambda, (a^b)):=\left(\lambda\cup\{a+b-1\}, \left((a-1)^{b-1}\right)\right).$$ 
    
    \item If $(\lambda, (a^b))\in A_2(n)$, then let $m$ denote the largest part of $\lambda$ that is not $a+b-1$ and define
    $$T(\lambda, (a^b)):=\left((\lambda\setminus\{m, a+b-1\})\cup(2a+2b-2+m), \left((a-1)^{b-1}\right)\right),$$ where $\lambda\setminus\{m, a+b-1\}$ is the partition obtained by removing parts $a+b-1$ and $m$  from $\lambda$. 
    
    \item If $(\lambda, (a^b))\in A_3(n)$, then $T(\lambda, (a^b)):=\left((a+1,a-1), \left((a+1)^{b-1}\right)\right)$. 
\end{enumerate}
 The image sets are thus 
\begin{align*}
    &T(A_1(n))=\{(\mu,(c^d))\in B(n) \ | \ c+d+1\in \mu\};\\
    &T(A_2(n))=\left\{(\mu,(c^d))\in B(n) \ \Big | \ \begin{array}{ll}c+d+1\not \in\mu, \, \mu_1\neq 3(c+d+1), \\ \text{and } \mu_1-\mu_2 >2(c+d+1)\end{array}\right\};\\
    &T(A_3(n))=\{(\mu,(c^d))\in B(n) \ | \ \mu=(c,c-2)\}.
\end{align*} 

Note that $T(A_1(n)),~T(A_2(n)),$ and $T(A_3(n))$ are disjoint, and their union $T(A(n))$ is a subset of $B(n)$. Define the map $L$ from $T(A(n))$ to $A(n)$ as follows:
\begin{enumerate}
    \item If $(\mu, (c^d))\in T(A_1(n))$, then $$L(\mu, (c^d)):=\left(\mu\setminus\{c+d+1\}, \left((c+1)^{d+1}\right)\right).$$ 
    
    \item If $(\mu, (c^d))\in T(A_2(n))$, then $$L(\mu, (c^d)):=\left((\mu\setminus\{\mu_1\})\cup\{c+d+1,\mu_1-2(c+d+1)\}, \left((c+1)^{d+1}\right)\right).$$ 
    
    \item If $(\mu, (c^d))\in T(A_3(n))$, then $$L(\mu, (c^d)):=\left((c+d-1), \left((c-1)^{d+1}\right)\right).$$ 
\end{enumerate}
Then   $L$ and $T$ are inverses of each other. Since $T$ gives a bijection between $A(n)$ and $T(A(n))\subseteq B(n)$, the excess in question is given by the number of elements in  
\begin{align*}
   B(n)&\setminus T(A(n))=B(n)\setminus(T(A_1(n))\cup T(A_2(n))\cup T(A_3(n)))\\
   &= \left\{(\mu,(c^d))\in B(n)\ \Big |\  \begin{array}{l}c+d+1\not \in \mu, \mu\neq (c,c-2)\text{, and}\\ \mu_1-\mu_2\leq 2(c+d+1)  \text{ or } \mu_1=3(c+d+1)\end{array}\right\},\\
   &= \left\{(\mu,(c^d))\in B(n)\ \Big |\   \mu \in \mathcal B(n,c,d) \right\}.
\end{align*} 
Theorem \ref{beck-lehmer} now follows.

\section{Proofs of Theorems \ref{lehmer-gen}, \ref{beck-lehmer2prime}, and \ref{beck-lehmer2}}\label{sec_proofs1}
For $r\in\mathbb N$, we define 
\begin{align*}
    F_r(z;q) &:= \frac1{(zq;q^{2r})_\infty(zq^2;q^{2r})_\infty \cdots (zq^{2r-1};q^{2r})_\infty\cdot(-zq^{2r};q^{2r})_\infty} \\
    &=\sum_{n=0}^\infty \sum_{m=0}^\infty \sum_{s=0}^m p\left(n \ \Bigg | \ \begin{array}{l}\text{all parts allowed,} \\m \text{ parts, }\\ s \text{ parts divisible by } 2r \end{array} \right)(-1)^s z^mq^n, \\ \ \\ 
R_r(z;q) &:=  \frac{(-zq^r;q^{2r})_\infty}{(zq;q^r)_\infty(zq^2;q^r)_\infty \cdots (zq^{r-1};q^r)_\infty} \\
&=\sum_{n=0}^\infty \sum_{m=0}^\infty  p\left(n \ \left| \ \begin{array}{l}\text{parts are not divisible by $2r$,}\\m\text{ parts,} \\ \text{parts divisible by $r$ are distinct} \end{array}\right.\right) z^m q^n.
\end{align*}
Hence, the generating series for $p_e(n,2r)-p_o(n,2r)$ and $q_o(n,r)$ are $F_r(1;q)$ and $R_r(1;q)$, respectively. 
We have 
\begin{align*}
F_r(1;q) 
&= \frac{ (q^{2r}; q^{2r})_\infty}{(q;q)_\infty} \cdot \frac1{(-q^{2r};q^{2r})_\infty }
= \frac{ (q^{2r}; q^{2r})_\infty}{(q;q)_\infty} \cdot \frac{(-q^r;q^{2r})_\infty}{(-q^r;q^r)_\infty}\\
&= \frac{ (q^{2r}; q^{2r})_\infty}{(q;q)_\infty} \cdot (q^r;q^{2r})_\infty \cdot (-q^r;q^{2r})_\infty 
= \frac{ (q^{r}; q^{r})_\infty}{(q;q)_\infty} \cdot  (-q^r;q^{2r})_\infty\\
&= R_r(1;q).
\end{align*}
Here we used the fact 
\[(-q^{2r};q^{2r})_\infty (-q^r;q^{2r})_\infty = (-q^r;q^r)_\infty \text{ and } (q^{2r};q^{2r})_\infty (q^r;q^{2r})_\infty = (q^r;q^r)_\infty \]
 in the second and fourth equality respectively, and used Euler’s identity 
 \[
 (-q;q)_\infty = \frac{1}{(q;q^2)_\infty}
 \]
 (by replacing $q$ by $q^r$) in the third equality. 
 Theorem \ref{lehmer-gen} now follows.
  
To prove Theorem \ref{beck-lehmer2prime},
  we have that  $\left.\frac{\partial}{\partial z}\right|_{z=1}(F_r(z;q)-R_r(z;q))$ is the generating series for the excess of the total number of parts in all partitions in $\mathcal{P}_e(n,2r)$ over the total number of parts in all partitions in $\mathcal{P}_o(n,2r)\cup \mathcal{Q}_o(n,r)$.
We have
\begin{align}
\notag   &\frac{\partial}{\partial z}\Big|_{z=1}(F_r(z;q)-R_r(z;q))  \\\notag 
   &= R_r(1;q)\Bigg(\sum_{\ell=1}^{2r-1}\sum_{k=0}^\infty\frac{q^{\ell+2kr}}{1-q^{\ell+2kr}}-\sum_{k=1}^\infty\frac{q^{2kr}}{1+q^{2kr}}
   \\\notag&{\hspace{1.35in}}-\sum_{\ell=1}^{r-1}\sum_{k=0}^\infty\frac{q^{\ell+kr}}{1-q^{\ell+kr}}-\sum_{k=0}^\infty\frac{q^{r+2kr}}{1+q^{r+2kr}}\Bigg) \\\notag
   &=R_r(1;q) \left( \sum_{\ell=1}^{2r-1}\sum_{k=0}^\infty\frac{q^{\ell+2kr}}{1-q^{\ell+2kr}}-\sum_{\ell=1}^{r-1}\sum_{k=0}^\infty\frac{q^{\ell+kr}}{1-q^{\ell+kr}}-\sum_{k=1}^\infty\frac{q^{kr}}{1+q^{kr}}\right) \\\notag
   &=R_r(1;q) \Bigg( \sum_{\ell=1}^{2r-1}\sum_{k=0}^\infty\frac{q^{\ell+2kr}}{1-q^{\ell+2kr}}-\sum_{\ell=1}^{r-1}\sum_{k=0}^\infty\frac{q^{\ell+kr}}{1-q^{\ell+kr}}\\\notag&\hspace{1.35in}-\sum_{k=1}^\infty\frac{q^{kr}}{1-q^{2kr}}+\sum_{k=1}^\infty\frac{q^{2kr}}{1-q^{2kr}}\Bigg) \\\notag
   &=R_r(1;q)\left(\sum_{k=1}^\infty \frac{q^k}{1-q^k} -\sum_{\ell=1}^{r-1}\sum_{k=0}^\infty\frac{q^{\ell+kr}}{1-q^{\ell+kr}}-\sum_{k=1}^\infty\frac{q^{kr}}{1-q^{2kr}} \right) \\\notag
   &=R_r(1;q)\left(\sum_{k=1}^\infty\frac{q^{kr}}{1-q^{kr}}-\sum_{k=1}^\infty\frac{q^{kr}}{1-q^{2kr}} \right) \\\label{eqn_Rr}
   &=R_r(1;q) \sum_{k=1}^\infty \frac{q^{2kr}}{1-q^{2kr}}.
\end{align}
This is the generating series for the number of pairs of partitions $(\lambda, (a^b))\vdash n$ so that 
\begin{itemize}
    \item[i.]  $2r\mid a$,
    \item[ii.] $\lambda\in \mathcal{Q}_o(n-ab,r)$.
\end{itemize}  Equivalently, \eqref{eqn_Rr} is the generating series for the number of partitions of $n$ in which among the parts divisible by $r$ there is exactly one even multiple of $r$, possibly repeated, and all other parts divisible by $r$ are odd multiples of $r$ and are distinct.
This proves Theorem \ref{beck-lehmer2prime}.

 To prove Theorem \ref{beck-lehmer2}, we define
\begin{align*}  
E_r(z;q) &:=
\frac1{ (q;q^{2r})_\infty (q^2;q^{2r})_\infty \cdots (q^{2r-1};q^{2r})_\infty  \cdot (-zq^{2r};q^{2r})_\infty }
 \\ &=\sum_{n=0}^\infty \sum_{m=0}^\infty p(n \ | \ \text{all parts allowed, } m \text{ parts divisible by } 2r  )(-z)^m q^n,
 \\  
Q_r(z;q)&:=\frac{(-z q^r;q^{2r})_\infty}{(q;q^r)_\infty(q^2;q^r)_\infty\dots(q^{r-1};q^r)_\infty} \\
&= \sum_{n=0}^\infty \sum_{m=0}^\infty p\left(n \ \Bigg | \ \begin{array}{l}\text{parts are not divisible by $2r$,} \\ \text{parts divisible by $r$ are distinct,} \\ 
m \text{ parts divisible by $r$}\end{array}\right) z^m q^n.
\end{align*}
As in the proof of Theorem \ref{beck-lehmer2prime}, we compute
\begin{align}\label{eqn_becklehmer}
\frac{\partial}{\partial z} \Big |_{z=1} (Q_r(z;q) - E_r(z;q)) = \frac{(- q^r;q^{2r})_\infty}{(q;q^r)_\infty(q^2;q^r)_\infty\cdots(q^{r-1};q^r)_\infty} \sum_{k=1}^\infty \frac{q^{kr}}{1+q^{kr}}.
\end{align}
In the proof of Theorem \ref{beck-lehmer}, we have shown that 
\begin{equation}
\label{eqn_r=1}
    (-q;q^2)_\infty \sum_{k=1}^\infty \frac{q^{k}}{1+q^k}
\end{equation}
is the generating series for the number of pairs of partitions $(\lambda, (a^b))\vdash n$ satisfying the following conditions: 
\begin{itemize}
    \item[i.]  
    $a,b$ are both odd,
    \item[ii.] $\lambda\in \mathcal Q_o\cap \mathcal B(n,a,b)$. 
\end{itemize}
For each $r\in \mathbb{N}$, replacing $q$ by $q^r$ in \eqref{eqn_r=1} implies that $$(-q^r;q^{2r})_\infty \sum_{k=1}^\infty \frac{q^{kr}}{1+q^{kr}}$$ is the generating series for the number of  pairs $(\lambda^{div}, ((ar)^b))\vdash n$
satisfying the following conditions: 
\begin{itemize}
    \item[i.]  
    $a,b$ are both odd,
    \item[ii.] $\lambda^{div}\in \mathcal Q_o(n-rab,r)\cap \mathcal B_r(n,a,b)$ and every part of $\lambda^{div}$ is divisible by $r$.
\end{itemize}
Theorem~\ref{beck-lehmer2} follows from equation~\eqref{eqn_becklehmer}.

\section{Proofs of Theorems \ref{lehmer-gen2},  \ref{beck-lehmer3prime}, and \ref{beck-lehmer3}}
\label{sec_proofs2}
For $r\in \mathbb N, L_r\subseteq \{2,4,\ldots,2r\}$ as in Section \ref{sec_intro}, we define
\begin{align*}
{E}_{r, L_r}(z;q) &:= \frac{1}{\displaystyle (q;q^2)_{\infty} \prod_{\ell \in   L_r}(-zq^{\ell};q^{2r})_\infty}\\
&= \sum_{n=0}^\infty \sum_{m=0}^\infty p\left(n \ \Bigg | \ \begin{array}{l} \text{all odd parts allowed,} \\ \text{even parts $\equiv \ell\spmod{2r}, \ell \in   L_r$,}\\ \text{$m$ even parts}\end{array}\right)(-z)^m q^n, 
\\ 
{F}_{r, L_r}(z;q) &:= \frac{1}{\displaystyle (zq;q^2)_{\infty} \prod_{\ell \in   L_r}(-zq^{\ell};q^{2r})_\infty}\\
&=\sum_{n=0}^\infty  \sum_{m=0}^\infty \sum_{s=0}^m p\left(n \ \Bigg | \ \begin{array}{l} \text{all odd parts allowed,}\\\text{even parts  $\equiv \ell\spmod{2r}, \ell \in   L_r$,}\\ \text{$m$ parts, $s$ even parts}\end{array}\right)(-1)^s z^m q^n,
\\ 
{Q}_{r,  L_r}(z;q)&:=\mathop{\prod_{j=1}^{2r}}_{j\not\in   L_r}(-zq^j;q^{2r})_\infty\\
&=\sum_{n=0}^\infty \sum_{m=0}^\infty p\left(n \ \Bigg| \ \begin{array}{l}\text{all odd parts allowed,}\\ \text{even parts  $\not\equiv \ell\spmod{2r},$   $\ell \in   L_r$,}\\ m\text{ parts, all distinct,}\\\end{array}\right) z^m q^n.
\end{align*}

Theorem \ref{lehmer-gen2} now follows  from the fact that $E_{r,L_r}(1;q) = Q_{r,L_r}(1;q)$, which is not difficult to obtain after a short calculation using Euler's identity.  
   
  The proof of Theorem \ref{beck-lehmer3prime} is similar to the proofs of Theorems \ref{beck-lehmer0} and \ref{beck-lehmer2prime}, and can be seen from 
\begin{align*} 
\frac{\partial}{\partial z} \Big |_{z=1} & ({Q}_{r, L_r}(z;q) - {F}_{r,  L_r}(z;q)) = Q_{r,L_r}(1;q) \sum_{k=1}^\infty \frac{q^{2k}}{1-q^{2k}}.  \end{align*}

To prove Theorem \ref{beck-lehmer3}, we compute

\begin{align} 
\frac{\partial}{\partial z} \Big |_{z=1} & ({Q}_{r, L_r}(z;q) - {E}_{r,  L_r}(z;q)) \nonumber\\&=\mathop{\prod_{j=1}^{2r}}_{j\not\in   L_r}(-q^j;q^{2r})_\infty  \left(\sum_{k=1}^\infty \frac{q^{k}}{1+q^{k}}\right) \nonumber\\
    &=
    \left(\mathop{\prod_{j=1}^{2r}}_{j \text{ even, } j \not\in   L_r}(-q^j;q^{2r})_\infty \right)   (-q;q^{2})_\infty   \sum_{k=1}^\infty \frac{q^{k}}{1+q^{k}}. \label{product}
\end{align}
Using the combinatorial interpretation of \eqref{eqn_r=1} in the proof of Theorem \ref{beck-lehmer}, Theorem~\ref{beck-lehmer3} follows from \eqref{product}.

\section{Proof of Theorem \ref{beck-lehmer4}}\label{sec_proofs3}

Let $r\in\mathbb N$, $L_r\subseteq\{2,4,\ldots,2r\}$ and $O_r\subseteq\{1,3,\ldots,2r-1\}$ as in Section \ref{sec_intro}.
Also let $L_r^c = \{2,4,\ldots, 2r\}\setminus L_r$ and $O_r^c = \{1,3,\ldots, 2r-1\}\setminus O_r$.
Define
\begin{align*}
\widetilde{E}_{r, L_r}(z;q) &:=
\frac{1}{\displaystyle (q;q^2)_\infty\prod_{j\in L_r^c}(-q^j;q^{2r})_\infty 
 \prod_{\ell\in L_r}(-zq^\ell;q^{2r})_\infty} \\
 &=\sum_{n=0}^\infty \sum_{m=0}^\infty (p_{e}(n,m;L_r)-p_{o}(n,m;L_r)) z^m q^n,
\\
\widetilde{Q}_{r, O_r}(z;q) &:=  \prod_{j \in  O_r^c}(-q^j;q^{2r})_\infty  \prod_{\ell \in  O_r}(-zq^\ell;q^{2r})_\infty\\
&=\sum_{n=0}^\infty \sum_{m=0}^\infty 
q_o(n,m;O_r) z^m q^n,
\end{align*}
where 
\begin{align*} p_{e/o}(n,m;L_r) &:= 
p\left(n \ \Bigg | \ \begin{array}{l} \text{all parts allowed,} \\
\text{even/odd no.\ of even parts,} \\ 
\text{$m$ (even) parts $\equiv \ell\spmod{2r}, \ell \in   L_r$}\end{array}\right), \\
q_o(n,m;O_r) &:= p\left(n \ \Big | \ \begin{array}{l} \text{all parts odd and distinct,} \\ 
\text{$m$ (odd) parts $\equiv \ell \spmod{2r}, \  \ell \in O_r$}\end{array}\right).
\end{align*}
When $z=1$, $\widetilde{E}_{r, L_r}(1;q)=\widetilde{Q}_{r, O_r}(1;q)$ recovers Lehmer’s identity \eqref{lehmer} in Theorem \ref{lehmer-thm}.

We compute that
\begin{align}\label{eqn_QEtilderiv}
    \frac{\partial}{\partial z} \Big |_{z=1} (\widetilde{Q}_{r, O_r}(z;q) - \widetilde{E}_{r,  L_r}(z;q)) &=(-q;q^2)_\infty \sum_{ \ell \in  L_r \cup   O_r}\mathop{\sum_{k=0}^\infty} \frac{q^{2kr+\ell}}{1+q^{2kr+\ell}}.
\end{align}

To prove Theorem \ref{beck-lehmer4}, it suffices to prove the case where $L_r\cup O_r=\{\ell\}$ for each positive integer $\ell\leq 2r$.
In Section \ref{sec_ec}, we state and prove Proposition \ref{lem_ec}, which establishes the non-negativity of the $q$-series coefficients of the series in  \eqref{eqn_QEtilderiv} (noting that special case $\ell=2$ is more delicate).   Then we provide two different proofs of Theorem \ref{beck-lehmer4}:  the first proof in Section \ref{sec_eccomb1} makes use of Proposition \ref{lem_ec} and its proof,  while the second proof in Section \ref{sec_eccomb2} establishes a relevant combinatorial  injection and is independent of Proposition \ref{lem_ec}.  

\subsection{Non-negativity of $q$-series coefficients}\label{sec_ec}
We use the notation $F(q)\succeq 0$ to mean that the coefficients of  $F(q)$ when expanded as a $q$-series are all non-negative.
\begin{proposition} \label{lem_ec} Let $r \in \mathbb N$, and $\ell$ a positive integer such that $\ell\leq 2r$.

If $\ell \neq 2$, then 
$$(-q;q^2)_\infty  \mathop{\sum_{k=0}^\infty} \frac{q^{2kr+\ell}}{1+q^{2kr+\ell}}\succeq 0.$$

If $\ell=2$, then the only possible negative coefficients of
\begin{align}\label{def_ell2} (-q;q^2)_\infty  \mathop{\sum_{k=0}^\infty} \frac{q^{2kr+2}}{1+q^{2kr+2}}\end{align}
(when expanded as a $q$-series) are the coefficients of $q^4, q^8, q^{12}, q^{16},$ and $q^{20}$, and any such negative coefficient is equal to $-1$.  Precisely,  
the set of all $n$ such that the coefficient of $q^n$ in \eqref{def_ell2} is negative (and thus  equal to $-1$) is given as a function of $r$ in the following table:

$$\begin{array}{|l|l|}  \hline
r &  \{n\}  \\ \hline 1 & \{ \ \}   \\ \hline
2,4,7 & \{4,8,12\} \\  \hline
3 & \{4\}\\  \hline
5 & \{4,8\}\\ \hline
6,9 & \{4,8,12,16\}\\ \hline
8, \text{or} \geq 10 & \{4,8,12,16,20\} \\ \hline
\end{array}$$
\end{proposition}
 
\begin{proof}[Proof of Proposition \ref{lem_ec}]   We divide our proof into three cases:  $\ell$ odd, $\ell$ even but $\ell \neq 2$, and $\ell=2$.

For $0<\ell\leq 2r$ odd, we have that  \begin{align}(-q;q^2)_\infty\sum_{k=0}^{\infty} \frac{q^{2kr+\ell}}{1+q^{2kr+\ell}}
=\sum_{k=0}^{\infty} \mathop{\prod_{m=0}^\infty}_{2m+1\neq 2kr+\ell}q^{2kr+\ell}(1+q^{2m+1}) \succeq 0 \label{leftoversGF}.
\end{align}  

For $0<\ell\leq 2r$ even, we note that
\begin{equation}
\begin{split}\label{eqn_cbconj}
    (-q;q^2)_\infty\sum_{k=0}^{\infty} \frac{q^{2kr+\ell}}{1+q^{2kr+\ell}}=&\sum_{k=0}^{\infty}(-q;q^2)_\infty(1-q^{2kr+\ell}) \frac{q^{2kr+\ell}}{1-q^{2(2kr+\ell)}}.
\end{split}\end{equation}

We first assume that $\ell\neq 2.$  Using \eqref{eqn_cbconj}, it suffices to show that
$(-q;q^2)_\infty (1-q^{2a}) \succeq 0$   for any integer $a\geq 2$.    We apply the well-known identity (see, e.g., \cite[(2.2.6) with $q \mapsto q^2, t\mapsto q$]{AndEncy})
$$ (-q;q^2)_\infty  = \sum_{n=0}^\infty \frac{q^{n^2}}{(q^2;q^2)_n}$$ to  re-write
\begin{align}\nonumber (-q;q^2)_\infty (1-q^{2a}) & = (1-q^{2a})\sum_{n=0}^\infty \frac{q^{n^2}}{(q^2;q^2)_n} \\\nonumber
&= (1-q^{2a}) + \frac{(1-q^{2a})}{(1-q^2)}\left( \sum_{n=1}^\infty \frac{q^{n^2}}{(q^4;q^2)_{n-1}}\right) \\\label{eqn_jevenan}
&= 1-q^{2a} + \left(\sum_{t=0}^{a-1} q^{2t} \right)\left(\sum_{n=1}^\infty \frac{q^{n^2}}{(q^4;q^2)_{n-1}}\right).
\end{align}  Thus, it suffices to show that the coefficient of $q^{2a}$  in the $q$-series expansion of \begin{align}\label{eqn_jevenexpproof} \left(\sum_{t=0}^{a-1} q^{2t} \right)\left(\sum_{n=1}^\infty \frac{q^{n^2}}{(q^4;q^2)_{n-1}}\right)\end{align} is strictly positive. We re-write $2a= u^2 + v$, where $u^2$ is the largest even perfect square at most equal to $2a$ (with $u$  a non-negative even integer), and $v$ is a non-negative integer.  Note that $u^2 \geq 4$ (so $u\geq 2$), since    $2a \geq 4$. Since $u$ is even, $v$ is even, and since $u^2\geq 4$, we have that $0\leq v \leq 2(a-2)$.     That is, $v=2t$ for some $0\leq t \leq a-2$. For this $t$, we consider  
\begin{align}\label{eqn_tconj}q^{2t} \sum_{n=1}^\infty \frac{q^{n^2}}{(q^4;q^2)_{n-1}},\end{align} which appears in \eqref{eqn_jevenexpproof}.  We extract the $n=u$ term 
${q^{u^2}}/{(q^4;q^2)_{u-1}}$
from the sum in \eqref{eqn_tconj}, noting that $u\geq 2$.    Expanding this as a $q$-series, we obtain  
\begin{align}\label{eqn_tconj2}  \frac{q^{u^2}}{(q^4;q^2)_{u-1}} = q^{u^2}+\Sigma_u(q),\end{align} where $\Sigma_u(q)\succeq 0$. 
Multiplying \eqref{eqn_tconj2} by $q^{2t}$, we find the term $q^{2t+u^2} = q^{2a}$ in the $q$-expansion of \eqref{eqn_jevenexpproof};  moreover, we have explained that $\Sigma_u(q) \succeq 0$, and it is clear that the remaining coefficients in the $q$-expansion of \eqref{eqn_jevenexpproof} are non-negative.  This  completes the proof of non-negativity  in the case of even $\ell \neq 2 $.

When $\ell=2$, we begin with the identity  in \eqref{eqn_jevenan}, which also holds for $a=1$. 
In this case, the $q$-series expansion for the expression in \eqref{eqn_jevenexpproof} is $q+q^4+O(q^8)$ (and has non-negative coefficients);  that is,   the coefficient of $q^2$ is   $0$.  Thus, the $q$-expansion of   \eqref{eqn_jevenan} in this case is
 $1+q-q^2+q^4+O(q^8)$, and has non-negative coefficients for all powers of $q$ greater than $4$.  Referring to \eqref{eqn_cbconj}, 
 as above, we have that 
 \begin{align}\label{eqn_jksum} \sum_{k=1}^{\infty}(-q;q^2)_\infty(1-q^{2kr+\ell}) \frac{q^{2kr+\ell}}{1-q^{2(2kr+\ell)}} \succeq 0 
    \end{align}  for any even $\ell \in L_r$, including $\ell=2$. 
 We also have from the above 
  that the $k=0$ term from \eqref{eqn_cbconj} satisfies
    \begin{align}\label{eqn_jkzero} (-q;q^2)_\infty(1-q^{\ell}) \frac{q^{\ell}}{1-q^{2\ell}}\succeq 0 \end{align}  for  even $\ell\geq 4$.  For $\ell=2$, we have shown that the only negative term appearing in the $q$-expansion  
    \begin{align}\label{eqn_j2exp}  (-q;q^2)_\infty(1-q^{2}) = 1+q-q^2+q^4+q^8+q^9+q^{12}+q^{13}+q^{15}+2q^{16}+O(q^{17}) \end{align} is $-q^{2}.$ 
    We multiply \eqref{eqn_j2exp}  by 
    \begin{align}\label{eqn_j2exp2}  \frac{q^{2}}{1-q^{4}} = q^2+q^6 + q^{10}+q^{14}+q^{18}+q^{22}+O(q^{26})\end{align} 
    (which  clearly has non-negative $q$-series coefficients) 
    to obtain 
    \begin{align}\notag q^2+q^3&-q^4+2 q^6+q^7-q^8+3 q^{10}+2 q^{11}-q^{12}+4 q^{14}\\\label{eqn_j2exp3}  &+3 q^{15}-q^{16}+q^{17}+6 q^{18}+4 q^{19}-q^{20} +  \Sigma(q),\end{align} where $\Sigma(q) = O(q^{21})$.  We now argue that $\Sigma(q) \succeq 0 $.  It is not difficult to see that the only  powers of $q$ in the expansion for $\Sigma(q)$ which may possibly have negative coefficients   are those $q^m$ such that  $m\equiv 0 \pmod 4$,  $m\geq 24$   (where we have also used that $\Sigma(q) = O(q^{21})$).      Now, 
   any $m\equiv 0 \pmod{4}$ such that $m\geq 24$ can also be written   as  $m=(3+1)^2+6+(2+4c)$ for some integer $c\geq 0$.   Thus, we also obtain the term 
   $+q^m$ after multiplying  \eqref{eqn_j2exp} and \eqref{eqn_j2exp2} as follows. We use the expression in \eqref{eqn_jevenan} (with $a=1$ and $t=0$) for \eqref{eqn_j2exp}, and take the numerator 
   $q^{(3+1)^2}$   of the $n=3$ term  and also $q^6$ from the expansion of the denominator $(q^4;q^2)_3$ of that same term. This yields a term $q^{(3+1)^2+6}$ after multiplying.   We now multiply by the term  $q^{2+4c}$  from the expansion of  $q^2/(1-q^4)$  in \eqref{eqn_j2exp2}.  Overall, this yields after multiplication the term $q^{(3+1)^2+6+(2+4c)}=q^m$, which cancels with the earlier $-q^m$.   This shows that  $\Sigma(q)\succeq 0$.  
    
   Thus, the only negative coefficients of the series in 
   \eqref{eqn_j2exp3} are  $q^{4}, q^8, q^{12}, q^{16},$ and $q^{20}$, and these coefficients are all equal to $-1$.  When added to the rest of the sum in \eqref{eqn_jksum} (which has non-negative coefficients), 
  this argument shows  that the only  powers of $q$ in the expansion of \eqref{def_ell2} (equivalently, \eqref{eqn_cbconj} with $\ell=2$) with potentially negative coefficients are  $q^4, q^8, q^{12}, q^{16}, q^{20}$, and that any such negative coefficient must be $-1$, as claimed.
   Moreover, for $r\geq 10$, and any $k\geq 1$, we have that   
   $$ \frac{q^{2kr+2}}{1-q^{2(2kr+2)}} = O(q^{22}),$$ which, when combined with the above argument, proves that the coefficients of $q^4, q^8, q^{12}, q^{16},$ and $q^{20}$ are all equal to $-1$.  The remaining negative coefficients as given in the table in Proposition \ref{lem_ec} for $1\leq r \leq 9$ are easily calculated directly. 
   
\end{proof}

\subsection{Combinatorial interpretation of Proposition \ref{lem_ec}}\label{sec_eccomb1}

In this section, we give a combinatorial interpretation of   the coefficient of $q^n$ in the $q$-series of Proposition \ref{lem_ec} in terms of the number of pairs of partitions $(\lambda, (a^b))\vdash n$  satisfying certain conditions.
This will complete the first proof of Theorem \ref{beck-lehmer4}.

Let $\ell$ be a positive integer such that $\ell\leq 2r$. 

If $\ell$ is odd, then  the coefficient of $q^n$ in  \eqref{leftoversGF}  is  the number of pairs of partitions $(\lambda, (a))\vdash n$   satisfying $a \equiv \ell \pmod{2r}$,  $\lambda\in \mathcal{Q}_o$, and   $a\not\in\lambda$.

If $\ell$ is even, $\ell \neq 2$, we substitute \eqref{eqn_jevenan} in \eqref{eqn_cbconj}  to obtain

 \begin{align} \nonumber  (-q;q^2)_\infty& \sum_{k=0}^{\infty} \frac{q^{2kr+\ell}}{1+q^{2kr+\ell}} =   \\   \label{I} & \sum_{\substack{a
 \equiv \ell\spmod{2r}\\ a > 0}}\left(1+\left(\sum_{t=0}^{\frac{a}{2}-1} q^{2t}\right)\left( \sum_{n=0}^\infty \frac{q^{(n+1)^2}}{(q^4;q^2)_{n}}\right)\right) \frac{q^{a}}{1-q^{2a}}\\ \nonumber - \\  
 \label{II}  &   \sum_{\substack{a
 \equiv \ell\spmod{2r}\\ a > 0}}q^{a} \frac{q^{a}}{1-q^{2a}}.  
 \end{align}
The $q$-series $$\sum_{n=0}^\infty \frac{q^{(n+1)^2}}{(q^4;q^2)_{n}}$$ is the generating series for the number of   self-conjugate partitions of $n\geq 1$ which are either $(1)$ or whose  smallest part is at least $2$. Thus, this is also  the generating series of the number of partitions of $n\geq 1$ into distinct odd parts which are either $(1)$ or partitions whose first two parts differ by exactly $2$. Moreover, $$\sum_{t=0}^{\frac{a}{2}-1} q^{2t}$$ is the generating series for partitions of $n$ into a single even part no larger than $a-2$. 

By adding a non-negative  even integer no larger than $a-2$ to $(1)$ or the first part of a partition into distinct odd parts whose first two parts differ by exactly $2$, we see that 
$$\left(\sum_{t=0}^{\frac{a}{2}-1} q^{2t}\right)\left( \sum_{n=0}^\infty \frac{q^{(n+1)^2}}{(q^4;q^2)_{n}}\right)$$
is the generating series for the number of partitions $\lambda$ of $n\geq 1$  into  distinct odd parts with $\lambda_1-\lambda_2\leq a$.  Note that $\lambda_2$ can be $0$.

Thus, \eqref{I} is the generating series for the number of pairs of partitions $(\lambda, (a^b))\vdash n$ with $b$ odd,  $a \equiv \ell \pmod{2r}$, and  $\lambda \in \mathcal  Q_o$  satisfying $\lambda_1-\lambda_2\leq a$. Note that $\lambda$ may be empty.

To interpret  \eqref{II}, for $a\geq 4$ and even, define $\mu(a)\in \mathcal Q_o(a)$ to be the partition  \begin{equation}\label{def_mu}
\mu(a):=\begin{cases} (\frac{a}{2}+1, \frac{a}{2}-1) & \mbox{ if $\frac{a}{2}$ is even}\\ (\frac{a}{2}+2, \frac{a}{2}-2) & \mbox{ if $\frac{a}{2}$ is odd}. 
\end{cases}
\end{equation}
Thus, \eqref{II} is the generating series for the number of pairs of partitions $(\mu(a), (a^b))\vdash n$ with $b$ odd,  $a \equiv \ell \pmod{2r}$.

Therefore,  the coefficient of $q^n$ in $$(-q;q^2)_\infty\sum_{k=0}^{\infty} \frac{q^{2kr+\ell}}{1+q^{2kr+\ell}}$$ is equal to  the number of the pairs of partitions $(\lambda, (a^b))\vdash n$ with $b$ odd,  $a \equiv \ell \pmod{2r}$, and  $\lambda \in \mathcal Q_o$ satisfying $\lambda_1-\lambda_2\leq a$ and $\lambda\neq\mu(a)$.

If $\ell=2$, the argument above fails only in the interpretation of $q^{a}\frac{q^{a}}{1-q^{2a}}$ when $a=2$. However, this $q$-series is just $\sum_{k=1}^\infty q^{4k}$. Thus, if $\ell=2$ and  $n \equiv 0\pmod 4$, the coefficient of $q^n$ in \eqref{eqn_cbconj} is one less than the number of pairs of partitions described above. If $n\geq 24$, $n \equiv 0 \pmod 4$, then $((9,7,5,1), (2^{(n-22)/2}))$ is a pair counted by the sequence whose generating series is \eqref{I}. Thus, if $n \not \in\{4,8,12,16,20\}$, the coefficient of $q^n$ in \eqref{eqn_cbconj} is non-negative and equal to the number of pairs of partitions $(\lambda, (a^b))\vdash n$ with $b$ odd,  $a \equiv \ell \pmod{2r}$, and $\lambda \in \mathcal Q_o$  satisfying $\lambda_1-\lambda_2\leq a$, $\lambda\neq \mu(a)$ and, if $n\equiv 0 \pmod 4$, also $(\lambda, (a^b))\neq ((9,7,5,1), (2^{(n-22)/2}))$.

If $\ell=2$, $r\geq 10$, and $n\in\{4,8,12,16,20\}$, then  $q^n$ appears only when $a=2$ in \eqref{I}.  If $b$ is odd, $n-2b \equiv 2\pmod 4$ and  $n-2b\leq 18$.  Thus, a partition $\lambda\vdash n-2b$ into distinct parts would have two parts. However, since  $n-2b \equiv 2\pmod 4$, it follows that $\lambda_1-\lambda_2\geq 4$. Therefore, there are no pairs of partitions $(\lambda, (a^b))$ counted by the sequence whose generating series is \eqref{I} and the coefficient of $q^n$ in \eqref{eqn_cbconj} is $-1$. For $r<10$, one can easily verify that the values of $n$  giving negative coefficients are as in the statement of the theorem.

\subsection{An alternate proof} \label{sec_eccomb2}
In this section, we provide an alternate proof of Theorem~\ref{beck-lehmer4}. This proof is independent of Proposition~\ref{lem_ec} and additionally proves the combinatorial interpretation in Section~\ref{sec_eccomb1}.

Recall that to prove Theorem \ref{beck-lehmer4}, it suffices to interpret \eqref{eqn_QEtilderiv} combinatorially in the case where $L_r\cup  O_r = \{\ell\}$. 
Then \eqref{eqn_QEtilderiv} becomes
\begin{equation}
\begin{split}
    &(-q;q^2)_\infty\sum_{k=0}^{\infty} \frac{q^{2kr+\ell}}{1+q^{2kr+\ell}}\\ \label{eqn_lsplit}
    =&(-q;q^2)_\infty\sum_{k=0}^{\infty} \frac{q^{2kr+\ell}}{1-q^{2(2kr+\ell)}}-(-q;q^2)_\infty\sum_{k=0}^{\infty} \frac{q^{2(2kr+\ell)}}{1-q^{2(2kr+\ell)}}\\
    =:&\sum_{n=0}^{\infty}c_{\ell,r}(n)q^n, 
\end{split}
\end{equation}

{

From \eqref{eqn_lsplit}  we see that $c_{\ell,r}(n)$ is the excess of the number of elements in $$B_{\ell,r}(n):=\left\{(\lambda,(a^b))\vdash n\ \Big |  \begin{array}{l} \lambda \in \mathcal Q_o,  b \text{ odd}, a\equiv \ell\spmod{2r}\end{array}\!\right\}$$
over that in
$$A_{\ell,r}(n):=\left\{(\lambda,(a^b))\vdash n\ \Big |  \begin{array}{l} \lambda \in \mathcal Q_o,  b \text{ even}, a\equiv \ell\spmod{2r}\end{array}\!\right\}.$$
We create an injection $T_{\ell,r}$ from $A_{\ell,r}(n)$ into $B_{\ell,r}(n)$ as follows. 

\textbf{If $\ell$ is odd:} For $(\lambda, (a^b))$ in $A_{\ell,r}(n)$, define $$T_{\ell,r}(\lambda, (a^b))= \begin{cases} (\lambda\cup \{a\},(a^{b-1}))  & \mbox{ if $a \not \in\lambda$},\\
(\lambda \setminus \{a\}, (a^{b+1}))  & \mbox{ if $a\in \lambda$}.
\end{cases}$$

 Then,  $c_{\ell,r}(n)=|B_{\ell,r}(n)\setminus T_{\ell}(A_{\ell,r}(n))|$, i.e., the number  of pairs of partitions $(\lambda, (a))\vdash n$ such that
\begin{enumerate}
    \item[(I$_o$)]  $a \equiv \ell \pmod{2r}$
    \item[(II$_o$)] $\lambda\in \mathcal{Q}_o$  such that $a$ is not a part of $\lambda$.    
\end{enumerate}

\textbf{If $\ell$ is even:} For $(\lambda, (a^b))$ in $A_{\ell,r}(n)$ with $(\lambda, a)\neq (\varnothing, 2)$, define 
$$T_{\ell,r}(\lambda, (a^b))= \begin{cases} (\lambda\setminus\{\lambda_1\}\cup \{\lambda_1+a\}, (a^{b-1})) & \mbox{ if } \lambda\neq\varnothing\\
(\mu(a), (a^{b-1})) & \mbox{ if } \lambda=\varnothing,
\end{cases}$$
where $\mu(a)$ was defined in \eqref{def_mu}.

When $n\not\equiv 0\pmod 4$ or $\ell\neq 2$, we have $(\lambda, a)\neq (\varnothing, 2)$ for all $(\lambda, (a^b))$ in $A_{\ell,r}(n)$. Then $c_{\ell,r}(n)=|B_{\ell,r}(n)\setminus T_{\ell,r}(A_{\ell,r}(n))|$, i.e., the number  of pairs of partitions $(\lambda, (a^b))\vdash n$ such that 
\begin{enumerate}
    \item[(I$_e$)] $a \equiv \ell \pmod{2r}$  and $b$ is odd. 
    \item[(II$_e$)] $\lambda\in \mathcal{Q}_o$ such that $\lambda_1-\lambda_2\leq a$, $\lambda \neq \mu(a)$.
\end{enumerate}

If $n\equiv 0\pmod 4$ and $\ell =2$, then $(\varnothing, 2^{n/2})\in A_{\ell,r}(n)$. If $n\geq 24$, we define 

\begin{equation*}\label{eq:imageT}
    T_{\ell,r}(\varnothing, (2^{n/2}))= ((9,7,5,1),(2^{(n-22)/2})).
\end{equation*}

  If $n\not \in \{4,8,12,16,20\}$, then $c_\ell(n)=|B_{\ell,r}(n)\setminus T_{\ell,r}(A_{\ell,r}(n))|$, i.e., the number  of pairs of partitions $(\lambda, (a^b))\vdash n$ satisfying i. and ii. above and the additional condition
\begin{enumerate}
    \item[(III$_e$)] If $a=2$, then $\lambda\neq (9,7,5,1)$.
\end{enumerate}

If $n\in \{4,8,12,16,20\}$ there is no obvious image of  $(\varnothing, 2^{n/2})$ under $T_{\ell,r}$ so that the transformation remains injective. In this case, $c_{\ell,r}(n)$ is one less than the number of pairs of partitions $(\lambda, (a^b))\vdash n$ satisfying i. and ii. above. Depending on $r$ and $n$, in this case $c_{\ell,r}(n)$ could be $-1$.

}

\begin{remark} 
For sufficiently large $n$, the coefficients of $q^n$ are strictly positive.    We explain  this more precisely below.   

Suppose $\ell$ is odd and $n\geq \ell+8$.  Let  $d:=n-\ell$.   Then,    if $d$ is even, we have $(d-1, 1), (d-3, 3)\in \mathcal Q_o(d)$, and if $d$ is odd, we have $(d), (d-4, 3,1)\in \mathcal Q_o(d)$. Since pairs of partitions in each case have disjoint sets of parts, there is a pair $(\lambda, (\ell))\vdash n$ satisfying (I$_o$) and (II$_o$).

Suppose $\ell$ is even and $n\geq \ell+19$, and let $d:=n-\ell$ as above. The following pairs $(\lambda, \ell)$ satisfy conditions (I$_e$) and (II$_e$).

If $d\equiv 0\pmod 4$, let $\lambda:=(\frac{d}{2}-1,\frac{d}{2}-3, 3,1)$. 

If $d\equiv 1\pmod 4$, let $\lambda:=(\frac{d+1}{2},\frac{d-3}{2},1)$. 

If $d\equiv 2\pmod 4$, let $\lambda:=(\frac{d}{2}-2,\frac{d}{2}-4, 5,1)$. 

If $d\equiv 3\pmod 4$, let $\lambda:=(\frac{d-1}{2},\frac{d-5}{2}, 3)$. 

If $n\geq 28$ and $d\equiv 2 \pmod 4$ with $\ell=2$, then  (III$_e$) is also satisfied.
\end{remark}

\subsection{Examples of Theorem \ref{beck-lehmer4}}
In this section, we give  some examples of Theorem \ref{beck-lehmer4}  for specific choices of $L_r$ and $O_r$  in which the excess is non-negative for all $n$.     Example 1 gives a new interpretation for the excess studied in Theorem \ref{beck-lehmer}. Examples 2 and 3, when specialized to $r=1$, are also related to the excess studied in Theorem \ref{beck-lehmer}.
\subsubsection{Example 1: $L_r\cup O_r = \{1,\ldots,2r\}$}
The excess in this case is the same as that in Theorem \ref{beck-lehmer}, but the combinatorial description given by Theorem \ref{beck-lehmer} and Theorem \ref{beck-lehmer4} are different.

When $n\not\equiv 0 \pmod 4$, Theorem \ref{beck-lehmer4} becomes:
The excess of the number of parts in all partitions in $\mathcal{Q}_o(n)$ plus the number of even parts in all partitions of $\mathcal{P}_e(n,2)$ over the number of even parts in all partitions in $\mathcal{P}_o(n,2)$ equals the number of pairs of partitions $(\lambda, (a^b))\vdash n$ satisfying the following conditions: 
\begin{itemize}
    \item[i.] $b$ is odd. Moreover, if $a$ is odd, then $b=1$,
    \item[ii.] $\lambda\in\mathcal{Q}_o$. Moreover, if $a$ is odd, then $a\not \in \lambda$; if $a$ is even, then $\lambda_1-\lambda_2\leq a$.
\end{itemize}

When $n\in\{4,8,12,16,20\}$, since $2\in L_r\cup O_r$, Theorem \ref{beck-lehmer4} does not guarantee the non-negativity of the excess because we could not define our injection on $(\varnothing,(2^{n/2}))$.
However, since $L_r\supseteq \{2,4,\ldots, 2r\}$, when $n\equiv 0\pmod{4}$, we can map $(\varnothing, (2^{n/2}))$ to $(\varnothing, (n) )$, proving non-negativity of the excess. Specifically, the excess is now equal to the number of pairs of partitions $(\lambda,(a^b))\vdash n$ satisfying the following conditions:
\begin{itemize}
    \item[i.] $b$ is odd. Moreover, if $a$ is odd, then $b=1$,
    \item[ii.] $\lambda\in\mathcal{Q}_o$. Moreover, if $a$ is odd, then $a\not \in \lambda$; if $a$ is even, then $\lambda_1-\lambda_2\leq a$ and $\lambda\neq\mu(a) \}$; if $a \equiv 0 \pmod 4$ and $b=1$, then $\lambda\neq\varnothing$.
\end{itemize}

With the modified injection, we have the following corollary of Theorem \ref{beck-lehmer4}:
\begin{corollary}\label{cor-even} The excess of the number of even parts in all partitions of $\mathcal P_o(n,2)$ over the  number of even parts  in all partitions of $\mathcal P_e(n,2)$ equals the number of partitions $\lambda$ of $n$ such that exactly one part is even, all other parts are odd and distinct, the even part may be repeated an odd number of times and, if we write $\lambda=(\lambda^o\cup ((2k)^b))$ with $\lambda^o\in\mathcal{Q}_o$, $k\geq 1$, then $\lambda^o_1-\lambda^o_2 \leq 2k, \lambda^o\neq\mu(2k)$, and if $k$ is even and $b=1$ then $\lambda^o\neq \varnothing$. \end{corollary} 
\begin{proof}
Corollary \ref{cor-even} follows from the fact that counting $(\lambda, (a))\vdash n$ with $a$ odd and $a\not \in \lambda$ is the same as counting parts in $\mathcal Q_o(n)$. When $a=2k$, we insert the even (possibly repeated) part into the partition into distinct odd parts  to obtain a statement similar to the original Beck conjecture.  
\end{proof}

Combining Theorem \ref{beck-lehmer0} and Corollary \ref{cor-even}, we arrive at the following corollary.

\begin{corollary} 
\label{cor-odd}
The excess of the number of odd parts in all partitions in $\mathcal P_e(n,2)$ over the number of odd parts in $\mathcal P_o(n,2)\cup \mathcal Q_o(n)$ equals the number of partitions $\lambda$ of $n$ satisfying either
\begin{itemize}
\item[i.] $\lambda$ has exactly one even part, possibly repeated, and all other parts are odd and distinct, or
\item[ii.] all parts of $\lambda$ are odd and exactly one part $b$ is repeated. Moreover, let $\lambda^o = \lambda \setminus (b^{2k})$ be the partition obtained by removing from $\lambda$ the largest even number of parts equal to $b$, then $\lambda^o_1-\lambda^o_2 \leq 2k, \lambda^o\neq\mu(2k)$, and if $k$ is even and $b=1$ then $\lambda^o\neq \varnothing$.
\end{itemize}
\end{corollary}
\begin{proof}
The excess in Corollary \ref{cor-odd} is the sum of the excess in Theorem \ref{beck-lehmer0} and Corollary \ref{cor-even}. The partitions described in \emph{ii.} are disjoint from the partitions described in \emph{i.} They are in one-to-one correspondence with the partitions described in Corollary \ref{cor-even}.  
To see this correspondence, consider a partition $\mu=(\mu^o \cup ((2k)^b))$ as in Corollary \ref{cor-even}. Then define $\lambda=\mu^o\cup (b^{2k})$.
Now $b$, which is odd, is a repeated part. The part $b$ may have already existed in $\mu^o$, so its multiplicity in $\lambda$ can be even or odd. 
\end{proof}

\subsubsection{Example 2: $L_r\cup O_r = \{r,2r\}$}
In this case,  
\[\frac{\partial}{\partial z}\Big|_{z=1}(\widetilde Q_{r,O_r}(z;q) - \widetilde E_{r,L_r}(z;q)) = (-q;q^2)_\infty \sum_{k=0}^\infty \frac{q^{kr}}{1+q^{kr}}.\]
Theorem \ref{beck-lehmer4} implies that the series has non-negative coefficients when $r\geq3$, because $2\not \in L_r\cup O_r.$
When $r=1$, both Example 1 and Proposition \ref{lem_ec} show that the series has non-negative coefficients.
When $r=2$, the series also has non-negative coefficients.
As in Example 1, when $n\equiv 0 \pmod 4$, we can map $(\varnothing, (2^{n/2}))$ to $(\varnothing, (n))$.
Alternatively, one can see that in Proposition \ref{lem_ec}, the coefficient of $q^n$ in $(-q;q^2)_\infty \sum_{k=0}^\infty \frac{q^{4k+2}}{1+q^{4k+2}}$ is negative (and equal to $-1$) only for $n=4,8,12$, while it can be computed that the coefficient of $q^4, q^8, q^{12}$ in $(-q;q^2)_\infty \sum_{k=0}^\infty \frac{q^{4k+4}}{1+q^{4k+4}}$ are $1,1,4$, respectively.

\subsubsection{Example 3: $L_r\cup O_r = \{1,2r\}$}
In this case,
\begin{align*}
&\frac{\partial}{\partial z}\Big|_{z=1}(\widetilde Q_{r,O_r}(z;q) - \widetilde E_{r,L_r}(z;q)) \notag \\
=& (-q;q^2)_\infty \left(\sum_{k=0}^\infty \frac{q^{2kr+1}}{1+q^{2kr+1}} +\sum_{k=1}^\infty \frac{q^{2kr}}{1+q^{2kr}}\right).\notag
\end{align*}
Theorem \ref{beck-lehmer4} (for $r\ge 2$) and Example 1 (for $r=1$) imply that the series has non-negative coefficients. 

\begin{remark}
 The  derivative differences given in Examples 2 and 3  also have interpretations as generating series for the number of pairs of partitions satisfying certain conditions as in Theorem \ref{beck-lehmer4}.
\end{remark}

\section{Further non-negativity results}\label{sec_aux}

The derivative difference in Example 3 can be expressed as
\begin{align}
&(-q;q^2)_\infty
\left(\sum_{k=1}^{\infty} \frac{q^{2kr}}{1-q^{4kr}}-\sum_{k=0}^{\infty} \frac{q^{2(2kr+1)}}{1-q^{2(2kr+1)}}\right) \label{seriesThm6.2}\\
&\quad +(-q;q^2)_\infty\left( \sum_{k=0}^{\infty} \frac{q^{2kr+1}}{1-q^{2(2kr+1)}}-\sum_{k=1}^{\infty} \frac{q^{4kr}}{1-q^{4kr}}\right). \label{seriesThm6.3}
\end{align}
In Theorem \ref{thm6.2}, we show that \eqref{seriesThm6.2} has non-positive coefficients, and, in Theorem \ref{thm6.3}, we show that \eqref{seriesThm6.3} has non-negative coefficients.

\begin{theorem}
\label{thm6.2}
For $r\in \mathbb{N}$, we have that $$(-q;q^2)_\infty\left(\sum_{k=0}^{\infty} \frac{q^{2(2kr+1)}}{1-q^{2(2kr+1)}}-\sum_{k=1}^{\infty} \frac{q^{2kr}}{1-q^{4kr}}\right)\succeq 0.$$ 
 
\end{theorem}

\begin{proof}

It suffices to construct an injection $T$ from 
$$A(n):=\{(\lambda,(a^b))\vdash n\ | \ \lambda \in \mathcal Q_o, a\equiv 0\spmod{2r}, b \text{ odd}\}$$
to 
$$B(n):=\{(\lambda,(a^b))\vdash n\ | \ \lambda \in \mathcal Q_o,  a\equiv 1\spmod{2r},b \text{ even}\}.$$

Let $(\lambda, (a^b))\in A(n)$.
Let $0\leq c<2r$ be the remainder of $b-1$ when divided by $2r$. Note that $c$ is even. 
We partition the set $A(n)$ into two disjoint subsets:
\begin{align*}
    &A_1(n):=\{(\lambda,(a^b))\in A(n) \ | \  \lambda\neq\varnothing \};\\
    &A_2(n):=\{(\lambda,(a^b))\in A(n) \ | \ \lambda=\varnothing\}.
\end{align*}
We define $T$ on each $A_i(n)$ in the following way.
\begin{enumerate}
    \item If $\lambda\neq \varnothing$, then $$T(\lambda, (a^b)) = (\lambda \setminus \{\lambda_1\} \cup \{\lambda_1+ab-(a-c)(b-c)\}, (b-c)^{a-c}).$$
    \item If $\lambda =\varnothing$, then $$T(\varnothing, (a^b)) = (\mu(ab-(a-c)(b-c)), (b-c)^{a-c}).$$
\end{enumerate}
The image sets are thus
\begin{align*}
    T(A_1(n)))=&\{(\mu,(x^y))\in B(n) \ | \
   \mu_1-\mu_2> (y+z)(x+z)-xy\},\\
    T(A_2(n))=&\{(\mu,(x^y))\in B(n)  \ | \
    \mu=\mu((x+z)(y+z)-xy) \},
\end{align*}
where $z$ is the remainder of $-y$ when divided by $2r$.

Note that $T$ maps $(\lambda, (a^b))\in A(n)$ with $b\equiv \ell \pmod {2r}$ to $(\mu, (x^y))\in B(n)$ with $y\equiv -\ell+1 \pmod {2r}$.
When $b\equiv 1\pmod{2r}$, $T(\varnothing, (a^b)) = (\varnothing, (b^a))\not\in T(A_1(n))$.
When $b\equiv \ell \not\equiv 1\pmod{2r}$, because $(y+z)(x+z)-xy\geq2(x+y)+4\geq4$, $T(A_1(n))$ and $T(A_1(n))$ are disjoint.

Define the map $L$ from $T(A(n))$ to $A(n)$ as follows:
\begin{enumerate}
    \item If $(\mu,(x^y))\in T(A_1)$, then 
     $$L(\mu, (x^y)) = (\mu \setminus \{\mu_1\} \cup \{\mu_1-(y+z)(x+z)+xy\}, ((y+z)^{x+z})).$$
     \item If $(\mu,(x^y))\in T(A_2)$, then
     $$L(\mu, (x^y)) = (\varnothing, ((y+z)^{x+z})).$$
\end{enumerate}
Then $L$ and $T$ are inverse to each other. 
Hence $T$ is an injection and Theorem \ref{thm6.2} follows.
\end{proof}
\begin{remark}
When $r=1$, the injection $T$ is the \emph{bijection} $(\lambda, (a^b))\mapsto (\lambda, (b^a))$, where the conjugation $(a^b)\mapsto (b^a)$ was used in the proof of Theorem \ref{beck-lehmer}.
\end{remark}

\begin{theorem}
\label{thm6.3}
For $r\in \mathbb N$, we have that $$(-q;q^2)_\infty \left(\sum_{k=0}^\infty \frac{q^{2kr+1}}{1-q^{2(2kr+1)}}-\sum_{k=1}^\infty \frac{q^{4kr}}{1-q^{4kr}}\right)\succeq 0.$$  

\end{theorem}
\begin{proof}[First Proof.]
Because the derivative difference in Example 3 has non-negative coefficients, Theorem \ref{thm6.3} follows from Theorem \ref{thm6.2}.
\end{proof}

\begin{proof}[Second Proof.] 
Alternatively, we can also prove Theorem  \ref{thm6.3} directly.
It suffices to construct an injection $T$ from   $$A(n):=\{(\lambda,(a^b))\vdash n\ | \ \lambda\in \mathcal Q_o,a\equiv 0\spmod {2r}, b \text{ even} \}$$ to  $$B(n):=\{(\lambda,(a^b))\vdash n \ | \ \lambda \in \mathcal Q_o, a\equiv 1\spmod{2r}, b \text{ odd}\}.$$ We partition the set $A(n)$ into three disjoint subsets: 
\begin{align*}
    &A_1(n):=\{(\lambda,(a^b))\in A(n) \mid a+(2r-1)(b-1) \not\in \lambda \};\\ \ \\ 
    &A_2(n) \\ &:={\small{\{(\lambda,(a^b))\in A(n) \mid a+(2r-1)(b-1) \in \lambda \text{ and $\lambda$ has at least two parts}\}};}\\ \ \\ 
    &A_3(n):=\{(\lambda,(a^b))\in A(n) \mid \lambda =(a+(2r-1)(b-1))   \}.
\end{align*}

We define $T$ on each $A_i(n)$ in the following way.

\begin{enumerate}
\label{injection_general}
    \item If $(\lambda, (a^b))\in A_1(n)$ (including the case where $\lambda$ is empty), then $$T(\lambda, (a^b)):=\left(\lambda\cup \{a+(2r-1)(b-1)\}, \left((a+1-2r)^{b-1}\right)\right).$$ 
    \item If $(\lambda, (a^b))\in A_2(n)$, let $m$ denote the largest part of $\lambda$ that is not $(a+(2r-1)(b-1))$. Then $T(\lambda, (a^b))$ equals  {\small{\begin{align*}\left((\lambda\setminus\{a\!+(2r\!-1)(b\!-1),m\})\cup\{2(a\!+(2r\!-1)(b\!-1))\!+m\}, \left((a\!+1\!-2r)^{b\!-1}\right)\right).\end{align*}}} 
    \item If $(\lambda, (a^b))\in A_3(n)$, then $$T(\lambda, (a^b)):=\left(\left(a+1,a+(2r-2)b-(2r-1)\right), \left((a+1)^{b-1}\right)\right).$$ 
\end{enumerate}
The image sets are thus 
\begin{align*}
    &T(A_1(n))=\{(\mu,(c^d))\in B(n)\mid c+(2r-1)(d+1) \in \mu \}, \\ \ \\
    &T(A_2(n))=\left\{(\mu,(c^d))\in B(n) \ \Bigg | \ \begin{array}{l}  c+(2r-1)(d+1) \not \in \mu, \\  \mu_1-\mu_2>2(c+(2r-1)(d+1)), \\
    \mu_1 \neq 3(c+(2r-1)(d+1))
     \end{array}\right\},\\ \  \\
    &T(A_3(n))=\{(\mu,(c^d))\in B(n)\mid \mu =  (c,c+(2r-2)d-2)\}.
\end{align*}

Note that when $r=1$, $(2r-2)d-2=-2<0$, and $2\leq 2(c+(2r-1)(d+1))$.
When $r>1$, $(2r-2)d-2>0$ and $(2r-2)d-2\leq 2(c+(2r-1)(d+1))$.
Hence $T(A_1(n)), T(A_2(n)), T(A_3(n))$ are pairwise disjoint. 

Define the map $L$ from $T(A(n))$ to $A(n)$ as follows:
\begin{enumerate}
    \item If $(\mu,(c^d))\in T(A_1(n))$, then 
$$L(\mu,(c^d)):=(\mu\setminus\{c+(2r-1)(d+1)\},((c+2r-1)^{d+1})).$$
    \item If $(\mu,(c^d))\in T(A_2(n))$, then  we define $L(\mu,(c^d))$ by
   {\small{\begin{align*}((\mu\setminus\{\mu_1\})\cup\{c\!+(2r\!-1)(d\!+1),\mu_1-2(c\!+(2r\!-1)(d\!+1))\},((c\!+2r\!-1)^{d+1})).\end{align*}}}
    \item If $(\mu,(c^d))\in T(A_3(n))$, then $$L(\mu,(c^d)):=\left((c-1+(2r-1)d),\left((c-1)^{d+1}\right)\right).$$
\end{enumerate}
Then $L$ and $T$ are inverses of each other. 
Hence, $T$ is an injection and Theorem \ref{thm6.3} follows.
\end{proof}
\begin{remark}
When $r=1$, the second proof of Theorem \ref{thm6.3} recovers the proof of Theorem \ref{beck-lehmer}.
\end{remark}

\section*{Acknowledgements}  The authors thank the Banff International Research Station (BIRS) and the Women in Numbers 5 (WIN5) Program. The third author is  partially supported by National Science Foundation Grant DMS-1901791. The fifth author 
is partially supported by a FRQNT scholarship by Fonds de Recherche du Qu\'ebec, and an ISM scholarship by Institut des Sciences Math\'ematiques.

\ \\ \textit{Department of Mathematics and Computer Science, College of the Holy Cross, Worcester, MA 01610 USA} \\
\texttt{cballant@holycross.edu}     
\smallskip \\ 
\textit{School of Mathematics, University of Minnesota, Twin Cities, 
127 Vincent Hall 206 Church St. SE, Minneapolis, MN 55455, USA} \\
\texttt{hburson@umn.edu}
 \smallskip \\
\textit{Department of Mathematics and Statistics, Amherst College, Amherst, MA 01002, USA} \\
\texttt{afolsom@amherst.edu}  
 \smallskip \\ 
 \textit{Department of Mathematics, University of California, Los Angeles, Math Sciences Building, 520 Portola Plaza, Box 951555, 
Los Angeles, CA 90095, USA} \\
\texttt{cyhsu@math.ucla.edu}
\smallskip \\
 \textit{Mathematics and Statistics, McGill University, Burnside Hall
805 Sherbrooke Street West,
Montreal, Quebec H3A 0B9, Canada} \\ 
\texttt{isabella.negrini@mail.mcgill.ca}
\smallskip \\
 \textit{Department of Mathematics, University of Wisconsin-Madison, 480 Lincoln Drive, Madison, WI 53706, USA} \\ 
\texttt{bwen25@wisc.edu}
\end{document}